\documentclass[12pt,a4paper]{article}

\usepackage[T1]{fontenc}
\usepackage[utf8]{inputenc}

\usepackage{lmodern}

\usepackage{amsmath,amssymb,amsthm}

\usepackage[a4paper,margin=3cm]{geometry}

\newtheorem{theorem}{Theorem}[section]

\begin{document}

\title{Computing Gamma(p/q) with Beta function values}

\date{\today}

\maketitle

\begin{abstract}
In this expository article we show explicitly how to recursively compute Gamma(p/q) in terms of Beta function values which in turn are Kontsevich-Zagier periods.
\end{abstract}

\section{Introduction}
It was shown 25 years ago that $\Gamma(p/q)^q$ is a Kontsevich-Zagier Period, where $p,q \in \mathbb{N}$. \\ We want to compute this explicitly and assume $1 \leq p \leq q$. \\ \\A KZ Period is, roughly speaking, the integral of an algebraic function taken over a semi-algebraic domain, with coefficients in $\mathbb{Q}$. \\ \\
Define the following variation for $s \in \mathbb{C}$ and fix the principal branch \\ of the logarithm once and for all.
\begin{equation}
I_s := 2 \cdot \int \limits_{0}^{1} \frac{dx}{\sqrt{1 - x^s}} = \frac{2}{s} \cdot B(\frac{1}{s}, \frac{1}{2}) \ , \ \Re(s) \geq 1
\end{equation}
\\ \\
The following three equations effect a computation chain for $[p] \in \mathbb{Z} / q \mathbb{Z}$:

\begin{enumerate}
\renewcommand{\labelenumi}{\Roman{enumi})}
    \item \begin{equation}
    \Gamma(s+1)= s \cdot \Gamma(s) \ , \ \forall s \in \mathbb{C} \setminus - \mathbb{N}_0
    \end{equation}
    \item \begin{equation}
    \Gamma(\frac{1}{2s}) = \sqrt{\frac{2s \cdot I_{2s} \cdot \Gamma(\frac{1}{s})}{2^{(1/s)}}} \ , \ \Re(s) \geq \frac{1}{2}
    \end{equation}
    \item \begin{equation}
    \Gamma(s) \cdot \Gamma(1-s)=\frac{\pi}{\sin(\pi s)} \ , \ \forall s \in \mathbb{C} \setminus \mathbb{Z}
    \end{equation}
\end{enumerate}

The second equation follows from the Legendre Duplication formula via the Beta function.

\newpage
\section{Examples}
Due to the second equation, we can assume that $q$ is odd, because it allows us to relate $\Gamma(\frac{p}{2q}) \leftrightarrow \Gamma(\frac{p}{q})$. For example:
\begin{equation}
\Gamma(\frac{1}{2}) = \sqrt{I_2} = \sqrt{\pi}
\end{equation} \\
Similarly for $q=4$:
\begin{equation}
\Gamma(\frac{1}{4}) = \sqrt{2 \cdot I_4 \sqrt{2 \cdot I_2}}
\end{equation} \\
Our first nontrivial example is $q=3$:
\begin{equation}
\Gamma(\frac{1}{3}) \to \Gamma(\frac{2}{3}) \to \Gamma(\frac{1}{3})
\end{equation}
So we calculate with $\mathrm{II})$ using $s := \frac{3}{2}$:
\begin{equation}
    \Gamma(\frac{1}{3}) = \sqrt{\frac{3 \cdot I_{3} \cdot \Gamma(\frac{2}{3})}{2^{(2/3)}}}
    \end{equation}
    and then use the Euler relation for $\Gamma(2/3)$ and solve for $\Gamma(1/3)$. \\ \\
Similarly for $q = 5$:
\begin{equation}
\Gamma(\frac{1}{5}) \to \Gamma(\frac{2}{5}) \to \Gamma(\frac{4}{5}) \to \Gamma(\frac{1}{5})
\end{equation} \\    
\begin{equation}
\Gamma(\frac{2}{5}) \to \Gamma(\frac{4}{5}) \to \Gamma(\frac{1}{5}) \to \Gamma(\frac{2}{5})
\end{equation}
or equivalently,
\begin{equation}
\Gamma(\frac{2}{5}) \to \Gamma(\frac{4}{5}) \to \Gamma(\frac{8}{5}) \to \Gamma(\frac{3}{5}) \to \Gamma(\frac{6}{5}) \to \Gamma(\frac{1}{5}) \to \Gamma(\frac{2}{5})
\end{equation} \\ \\
We can already see here that there are multiple ways to get a recursive computation chain that closes back to the starting value. \\
However, there is a unique minimal i.e. shortest computation chain. \\ \\
For $q = 7$ we have the first example where, after substitution, a meromorphic differential form appears with the value $I_{7/4}$:
\begin{equation}
\Gamma(\frac{1}{7}) \to \Gamma(\frac{2}{7}) \to \Gamma(\frac{4}{7}) \to \Gamma(\frac{8}{7}) \to \Gamma(\frac{1}{7})
\end{equation}
A final example related to a specific value of an elliptic L-function:
\begin{equation}
\Gamma(\frac{3}{7}) \to \Gamma(\frac{6}{7}) \to \Gamma(\frac{12}{7}) \to \Gamma(\frac{5}{7}) \to \Gamma(\frac{10}{7})\to \Gamma(\frac{3}{7})
\end{equation}

\newpage
\begin{theorem}
For every odd $q$, there is a recursive computation chain that always closes back to $\Gamma(\frac{p}{q})$ after finitely many steps.
\end{theorem}
\begin{proof}
Let the allowed operations on \(p \pmod q\) be
\begin{enumerate}
\renewcommand{\labelenumi}{\Roman{enumi})}
\item
\begin{equation}
p \mapsto p \pmod q
\end{equation}
\item
\begin{equation}
p \mapsto 2p \pmod q
\end{equation}
\item
\begin{equation}
p \mapsto q-p \equiv -p \pmod q
\end{equation}
\end{enumerate}
where $q$ is odd.

Starting from \(2p \pmod q\), every sequence of operations produces a residue of the form
\[
\pm 2^n p \pmod q
\]
for some \(n\ge1\).

Let \(d=\gcd(p,q)\) and \(q'=q/d\). To return to \(p \pmod q\), we need
\[
\pm 2^n p \equiv p \pmod q,
\]
which is equivalent to
\[
\pm 2^n \equiv 1 \pmod {q'}.
\]

Since \(q'\) is odd, \(2\) is invertible modulo \(q'\). Hence \(2\) has finite order in
\[
(\mathbb Z/q'\mathbb Z)^\times,
\]
so there exists \(m>0\) such that
\[
2^m \equiv 1 \pmod {q'}.
\]
Therefore
\[
2^m p \equiv p \pmod q.
\]

So, after finitely many applications of rule (2), starting from our first operation \(2p \pmod q\), we return to \(p \pmod q\).
\end{proof}

\newpage
\section{Conjectures on Riemann surface interpretation}
In the rational case we have the following geometric interpretation for $n \geq 3$:
\begin{equation}
I_{n/k} = 2k \cdot \int \limits_{0}^{1} \frac{x^{k-1}}{\sqrt{1 - x^n}} \ dx = k \cdot \int \limits_{\gamma} \omega_{k-1}
\end{equation}
are integrals over the non-closed path $\gamma$ on the compact Riemann surface $X_n$ of the curve \\
\begin{equation}
C_n : y^2 = 1 - x^n \ , \ (x,y) \in \mathbb{C}^2
\end{equation} \\
where originally $x: 0 \to 1 \to 0$ and so $\gamma: P(0,1) \to P(1,0) \to P(0,-1)$ where we changed sheets at the branch point $1$. \\ \\
We see that $I_{n/k}$ is a Kontsevich-Zagier Period for $1 \leq k \leq n$ (for $k > n$ it can be reduced trivially), which follows from the integral representation, but \\not an integral of a holomorphic differential form over a closed cycle. \\ \\
The genus of $X_n$ is $g = \lfloor \frac{n-1}{2} \rfloor$. \\ \\
For $1 \leq k \leq g$ the differential form is holomorphic and for $g < k \leq n$ it is \\ meromorphic with a pole at infinity, which appeared with $I_{7/4}$ for $\Gamma(1/7)$. \\ \\
Finally, the integrals $I_{n/k}$ can (conjecturally) be seen as "basic building blocks" that may express genuine periods of $X_n$: \\ \\
Exploiting the symmetry of the branch points $\zeta_{n}^{m} = \exp(\frac{2 \pi i m}{n})$ we have (assuming a straight line connecting both endpoints): \\
\begin{equation}
\int_{1}^{\zeta_n}\frac{dx}{\sqrt{1 - x^n}} = (1 - \zeta_n) \cdot \int_{0}^{1}\frac{dx}{\sqrt{1 - x^n}}
\end{equation} \\ \\
These types of branch point integrals are used in the $A$- and $B$ period  matrix of $X_n$, but the RHS may need some correction factors due to square root, sheet changes and the fact that Cauchy's integral theorem does not hold in general. \\ \\
A reduction of $I_n$ to $\alpha \cdot I_d$ where $d$ divides $n$ and $\alpha$ an algebraic number, could therefore (conjecturally) be related to the decomposition of the Jacobian $J_n$ over a suited base field.

\newpage
\section{Table of Values (minimal computation chain)}

We can see here very clearly (and beautifully) that $\Gamma(p/q)^q$ is a Kontsevich-Zagier period. \\ \\
Additionally the $I_{n/k}$ are transcendental, which follows from the Beta function \\ representation and Schneider's theorem that for $a,b \in \mathbb{Q}$ but $a \notin \mathbb{Z}$ and $b \notin \mathbb{Z}$ and $a + b \notin \mathbb{Z}$, then $B(a,b)$ is transcendental.\\ \\
This is interesting, because if $\Gamma(1/q)$ is known to be transcendental, it would imply that a specific product of "independent" transcendental constants is itself transcendental. \\

\begin{table}[h!]
\centering

\renewcommand{\arraystretch}{2.4}

\begin{tabular}{|c|c|c|c}
\hline

\Large $q$
&
\Large $\Gamma\!\left(\frac{1}{q}\right)$
&
\Large \text{Elliptic Period K(k)?}
&
\Large \text{Meromorphic Form?}
\\

\hline

\Large $2$
&
\Large $\sqrt{I_2}$
&
 \textbf{Yes}
&
\textbf{No}
\\

\hline

\Large $3$
&
\Large
$\displaystyle
\left(
2^{1/3}\sqrt{3}\,I_2\, I_3
\right)^{1/3}
$
&
 \textbf{Yes}
&
\textbf{No}
\\

\hline

\Large $4$
&
\Large
$\displaystyle
\sqrt{2 I_4 \sqrt{2 I_2}}
$
&
 \textbf{Yes}
&
\textbf{No}
\\

\hline

\Large $5$
&
\Large
$\displaystyle
\left(
\frac{
125 I_2 \, I_5^{\,2}\, I_{5/2}
}{
2^{13/5}\sin(\pi/5)
}
\right)^{1/5}
$
&
 \textbf{???}
&
\textbf{No}
\\

\hline
\Large $6$
& 
\Large \text{use II) on Gamma(1/3)}
&
\textbf{Yes}
&
\textbf{No} \\
\hline

\hline
\Large $7$
&
\Large 
$\displaystyle
\left(
\frac{
7^6\,I_7^4 I_{7/2}^2 I_{7/4}
}{
2^{52/7}
}
\right)^{1/7}$
&
\textbf{???}
&
\textbf{Yes} \\
\hline
\Large $8$
&
\Large
$\displaystyle \sqrt{4 I_8 \sqrt{4 I_4 \sqrt{4 I_2}}}$
&
\textbf{Yes}
&
\textbf{No} \\
\hline
\end{tabular}
\end{table}

1) "Elliptic Period K(k)" asks if there is a description using the \\Elliptic K function for algebraic modulus $k$ with algebraic numbers as factors and possibly using $n$-th roots. \\ \\
2) "Meromorphic Form" asks if at least one $I_{n/k}$ appears with $k > g$.

\newpage
\section{Further Reading}

\end{document}